\documentclass[12pt,reqno]{amsart}

\usepackage{amsmath,amssymb,amsfonts,amsthm,latexsym,graphicx,multirow,hyperref,enumerate}
\usepackage[all]{xy}

\newcommand\A{\mathrm{A}}   \newcommand\Alt{\mathrm{Alt}} \newcommand\Aut{\mathrm{Aut}}
 
\newcommand\C{\mathrm{C}} \newcommand\calB{\mathcal{B}} \newcommand\calD{\mathcal{D}} \newcommand\calO{\mathcal{O}} \newcommand\Cay{\mathrm{Cay}}   \newcommand\Cos{\mathsf{Cos}}
\newcommand\D{\mathrm{D}}  
\newcommand\E{\mathrm{E}}
  \newcommand\Fix{\mathrm{Fix}}
\newcommand\G{\mathrm{G}}

 \newcommand\magma{{\sc Magma} }

\newcommand\Nor{\mathbf{N}}

       \newcommand\PSL{\mathrm{PSL}}    \newcommand\PSp{\mathrm{PSp}} \newcommand\PSU{\mathrm{PSU}}

\newcommand\SL{\mathrm{SL}}  \newcommand\Soc{\mathrm{Soc}}     \newcommand\Sy{\mathrm{S}} \newcommand\Sym{\mathrm{Sym}}

\newtheorem{theorem}{Theorem}[section]
\newtheorem{lemma}[theorem]{Lemma}
\newtheorem{proposition}[theorem]{Proposition}

\newtheorem{problem}[theorem]{Problem}
\newtheorem{question}[theorem]{Question}

\theoremstyle{definition}

\newtheorem{construction}[theorem]{Construction}

\begin{document}

\title[Tetravalent half-arc-transitive graphs]{Tetravalent half-arc-transitive graphs with unbounded nonabelian vertex stabilizers}

\author[Xia]{Binzhou Xia}
\address{School of Mathematics and Statistics\\The University of Melbourne\\Parkville, VIC 3010\\Australia}
\email{binzhoux@unimelb.edu.au}

%\date\today

\begin{abstract}
Half-arc-transitive graphs are a fascinating topic which connects graph theory, Riemann surfaces and group theory. Although fruitful results have been obtained over the last half a century, it is still challenging to construct half-arc-transitive graphs with prescribed vertex stabilizers. Until recently, there have been only six known connected tetravalent half-arc-transitive graphs with nonabelian vertex stabilizers, and the question whether there exists a connected tetravalent half-arc-transitive graph with nonabelian vertex stabilizer of order $2^s$ for every $s\geqslant3$ has been wide open. This question is answered in the affirmative in this paper via the construction of a connected tetravalent half-arc-transitive graph with vertex stabilizer $\mathrm{D}_8^2\times\mathrm{C}_2^m$ for each integer $m\geqslant1$, where $\mathrm{D}_8^2$ is the direct product of two copies of the dihedral group of order $8$ and $\mathrm{C}_2^m$ is the direct product of $m$ copies of the cyclic group of order $2$. The graphs constructed have surprisingly many significant properties in various contexts.

\textit{Key words:} tetravalent; half-arc-transitive

\textit{MSC2010:} 20B25, 05C25
\end{abstract}

\maketitle

\section{Introduction}

Throughout this paper all graphs are assumed to be finite, simple and undirected. An \emph{arc} of a graph is an ordered pair of adjacent vertices of the graph. Let $\Gamma$ be a graph and let $X$ be a subgroup of the automorphism group $\Aut(\Gamma)$ of $\Gamma$. We say that $X$ is \emph{vertex-transitive}, \emph{edge-transitive} or \emph{arc-transitive} if $X$ acts transitively on the vertex set, edge set or arc set, respectively, of $\Gamma$. If $X$ is vertex-transitive and edge-transitive but not arc-transitive, then we say that $X$ is \emph{half-arc-transitive} and $\Gamma$ admits a \emph{half-arc-transitive group} $X$. The graph $\Gamma$ is said to be half-arc-transitive if $\Aut(\Gamma)$ is half-arc-transitive.

The study of half-arc-transitive graphs dates back to the 1960s, if not earlier, when Tutte in a book of himself proved that the valency of a half-arc-transitive graph must be even~(\cite[7.53]{Tutte1966}). In the same book he pointed out that it was not known whether such graphs exist~(\cite[p.~60]{Tutte1966}). This was solved a few years later by Bouwer~\cite{Bouwer1970}, who constructed half-arc-transitive graphs of valency $2k$ for all integers $k\geqslant2$ (see also~\cite{CZ2016}). Note that a graph of valency $2$ is a disjoint union of cycles and hence not half-arc-transitive. Thus the valency of a half-arc-transitive graph is at least $4$.

Over the last half a century, numerous papers have been published on half-arc-transitive graphs, for which the reader may refer to the survey papers~\cite{CPS2015,Marusic1998i} and recent papers~\cite{CC2018,FGLPRV2018,GHX2019,JMSV2019,RS2019}. On the one hand, it is not surprising that the majority of papers deal with the case of valency $4$, the smallest valency of half-arc-transitive graphs. On the other hand, surprisingly rich structures and connections to various branches of mathematics (a somewhat surprising example is maps on Riemann surfaces~\cite{MN1998}) have been revealed during the investigation of tetravalent half-arc-transitive graphs, so that Maru\v{s}i\v{c}~\cite[p.221]{Marusic1998i} described progress on studying them as ``thrilling''.

In the extensive study of tetravalent half-arc-transitive graphs, a large amount of work has been devoted to the construction of such graphs. One reason is that although it is fairly easy to construct tetravalent graphs admitting a half-arc-transitive group, constructing certain half-arc-transitive graphs often turns out to be challenging. For valency greater than $4$, it is easy to verify that the vertex stabilizers in Bouwer's examples are nonabelian (see~\cite{RS2017,ZZ2018} for the automorphism groups of these graphs and their natural generalizations). However, in the case of valency $4$, there was a long time for which all known examples of connected half-arc-transitive graphs had abelian vertex stabilizers. See~\cite{AMN1994,MM1999,Marusic2005,MX1997,Sajna1998} for some of these examples, and see~\cite{MN2001} for a description of the structure of vertex stabilizers in tetravalent half-arc-transitive graphs.

The first known example of a connected tetravalent half-arc-transitive graph with nonabelian vertex stabilizer was found by Conder and Maru\v{s}i\v{c}~\cite{CM2003} about twenty years ago. The graph has order $10752$ and vertex stabilizer $\D_8$, the dihedral group of order $8$. Inspired by a comment made by Maru\v{s}i\v{c} about this graph being somewhat unique, in a lecture at a workshop at the Fields Institute in October 2011, considerable effort has been made to enlarge the list of examples of half-arc-transitive graphs with nonabelian vertex stabilizers.

In 2015, Conder, Poto\v{c}nik and \v{S}parl~\cite{CPS2015} constructed another connected tetravalent half-arc-transitive graph with vertex stabilizer $\D_8$ and order $10752$. This graph and the previous graph of the same order turn out to be the only connected tetravalent half-arc-transitive graphs with vertex stabilizer $\D_8$ and order $10752$, with $10752$ being the smallest order of connected tetravalent half-arc-transitive graphs with vertex stabilizer $\D_8$ (see~\cite{PP2017}). Also constructed in~\cite{CPS2015} are two other connected tetravalent half-arc-transitive graphs with nonabelian vertex stabilizers, one with vertex stabilizer $\D_8$ and order $21870$ and the other with vertex stabilizer $\D_8\times\C_2$ and order $90\cdot3^{10}$, and no other half-arc-transitive graph with nonabelian vertex stabilizer was known up to~2015. Thus the authors in~\cite{CPS2015} asked the following question.

\begin{question}\label{que1}
Does there exist a connected tetravalent half-arc-transitive graph with nonabelian vertex stabilizer of order $2^s$, for every $s\geqslant3$?
\end{question}

Despite consistent effort made, the above question has been wide open so far. In fact, it has been even unknown whether there exist infinitely many connected tetravalent half-arc-transitive graphs with nonabelian vertex stabilizers. The only new examples since 2015 were given by Spiga~\cite{Spiga2016}, who constructed a connected tetravalent half-arc-transitive graph with vertex stabilizer $\D_8^2$ and with vertex stabilizer a nonabelian group of order $128$, respectively. Although the constructions were not successfully extended to any infinite family, Spiga commented in~\cite{Spiga2016} that they are believed to be ``only the tip of an iceberg'' (see~\cite{Conder2019,DJ2013,PS2019} for evidence).

In this paper, we construct a connected tetravalent half-arc-transitive graph with vertex stabilizer $\D_8^2\times\C_2^m$ for each integer $m\geqslant1$. Together with the above mentioned connected tetravalent half-arc-transitive graphs with vertex stabilizers $\D_8$, $\D_8\times\C_2$ and $\D_8^2$, this gives the affirmative answer to Question~\ref{que1} for $s\neq5$. For $s=5$, it is easy to construct a connected tetravalent graph admitting a half-arc-transitive action of $\A_{32}$ with vertex stabilizer $\D_8\times\C_2^2$, and then~\cite[Theorem 1.1]{SX} shows the existence of its covers that are connected tetravalent half-arc-transitive graphs with vertex stabilizer $\D_8\times\C_2^2$. In fact, the very recent paper~\cite{SX} followed this approach to show the existence of connected tetravalent half-arc-transitive graphs with vertex stabilizer $\D_8\times\C_2^m$ for all $m\geqslant1$, which can be also used to answer Question~\ref{que1}. Before stating our main result, we introduce some terminology.

Let $G$ be a group and let $S$ be an inverse-closed nonempty subset of $G\setminus\{1\}$. The \emph{Cayley graph} $\Cay(G,S)$ on $G$ with \emph{connection set} $S$ is defined to be the graph with vertex set $G$ such that $x,y\in G$ are adjacent if and only if $yx^{-1}\in S$. Denote by $R_G(G)$ the subgroup of $\Sym(G)$ induced by the right multiplication of $G$ on itself, and
\[
\Aut(G,S)=\{\alpha\in\Aut(G)\mid S^\alpha=S\}
\]
the setwise stabilizer of $S$ in $\Aut(G)$. Then $R_G(G)\rtimes\Aut(G,S)$ is a subgroup of $\Aut(\Cay(G,S))$. The Cayley graph $\Cay(G,S)$ is said to be \emph{normal} if $R_G(G)$ is a normal subgroup of $\Aut(\Cay(G,S))$, and is said to be \emph{nonnormal} otherwise.

Let $\Gamma$ be a tetravalent half-arc-transitive graph. It is easy to see that the action of $\Aut(\Gamma)$ on the arc set has two orbits, denoted by $\calO_1(\Gamma)$ and $\calO_2(\Gamma)$. For each $i\in\{1,2\}$ and each edge of $\Gamma$, the orbit $\calO_i(\Gamma)$ contains exactly one of the two arcs corresponding to this edge. Hence $\calO_1(\Gamma)$ and $\calO_2(\Gamma)$ give two opposite orientations of the edges of $\Gamma$, and thus give two digraphs with the same vertex set as $\Gamma$, denoted by $\calD_1(\Gamma)$ and $\calD_2(\Gamma)$. A cycle of $\Gamma$ is called an \emph{alternating cycle} if consecutive edges along the cycle have opposite orientations in $\calD_1(\Gamma)$ (and thus also opposite orientations in $\calD_2(\Gamma)$). In a seminal paper~\cite{Marusic1998ii} in 1998, Maru\v{s}i\v{c} proved that all alternating cycles of the tetravalent half-arc-transitive graph $\Gamma$ have the same length, half of which is called the \emph{radius} of $\Gamma$, and any two alternating cycles of $\Gamma$ with nonempty intersection share the same number of vertices. This number is called the \emph{attachment number} of $\Gamma$, and $\Gamma$ is said to be \emph{loosely attached} if this number is $1$. Analyzing alternating cycles and attachment numbers has been a general approach to study tetravalent half-arc-transitive graphs (see for instance~\cite{Marusic1998ii,MP1999,MW2000,PS2017,Sparl2008}).

Let $\Delta$ be a digraph. The \emph{reverse} of $\Delta$ is the digraph obtained by reversing each arc of $\Delta$. We say that $\Delta$ is \emph{self-reverse} if the reverse of $\Delta$ is isomorphic to $\Delta$. For a positive integer $s$, an \emph{$s$-arc} of $\Delta$ is a tuple $(v_0,v_1,\dots,v_s)$ of vertices such that $(v_i,v_{i+1})$ is an arc of $\Delta$ for each $i\in\{0,1,\dots,s-1\}$. We say that $\Delta$ is \emph{$s$-arc-transitive} if $\Aut(\Delta)$ acts transitively on the set of $s$-arcs of $\Delta$. Note that if $\Delta$ is vertex-transitive and $s$-arc-transitive with $s\geqslant2$ then it is also $(s-1)$-arc-transitive.

Now we state the main result of this paper.

\begin{theorem}\label{thm1}
For every positive integer $m$, there exists a graph $\Gamma_m$ satisfying the following:
\begin{enumerate}[{\rm (a)}]
\item $\Gamma_m\cong\Cay(\A_{2^{m+6}-1},S)$ is a connected nonnormal Cayley graph on the alternating group $\A_{2^{m+6}-1}$ for some $S\subset\A_{2^{m+6}-1}$ such that $\Aut(\A_{2^{m+6}-1},S)=1$;
\item $\Gamma_m$ is a loosely attached tetravalent half-arc-transitive graph of radius $6$;
\item $\Aut(\Gamma_m)\cong\A_{2^{m+6}}$;
\item $\Aut(\Gamma_m)$ has vertex stabilizer isomorphic to $\D_8^2\times\C_2^m$;
\item $\calD_1(\Gamma_m)$ and $\calD_2(\Gamma_m)$ are both $(m+6)$-arc-transitive digraphs that are not self-reverse.
\end{enumerate}
\end{theorem}

The graph $\Gamma_m$ is constructed in Section~\ref{sec2} (Construction~\ref{con1}), with its properties shown in later sections (in particular, a proof of Theorem~\ref{thm1} is given at the end of Section~\ref{sec5}). Unlike the previously known examples of tetravalent half-arc-transitive graphs with nonabelian vertex stabilizers, where the proof of their properties involves more or less computer computation, the proof of Theorem~\ref{thm1} is not computer-assisted in this paper. Nevertheless, the proof of some properties of $\Gamma_m$ is based on tedious (but direct) calculation, and computer computation can be used to verify these properties for small values of $m$. For this purpose, codes in \magma~\cite{BCP1997} are given in the appendix of~\cite{Xia}. Besides answering Question~\ref{que1}, the properties of $\Gamma_m$ have many other significances. This is illustrated in the next section.

\section{More background}

\subsection{Normality of Cayley graphs and graphical regular representations}

As mentioned above, for each Cayley graph $\Gamma=\Cay(G,S)$, the automorphism group $\Aut(\Gamma)$ of $\Gamma$ contains the subgroup $R_G(G)\rtimes\Aut(G,S)$. It was shown by Godsil~\cite{Godsil1981} that this subgroup is actually the normalizer of $R_G(G)$ in $\Aut(\Gamma)$, that is,
\begin{equation}\label{eq35}
\Nor_{\Aut(\Gamma)}(R_G(G))=R_G(G)\rtimes\Aut(G,S).
\end{equation}
A Cayley graph $\Cay(G,S)$ is called a \emph{graphical regular representation} (\emph{GRR} for short) of $G$ if $\Aut(\Cay(G,S))=R_G(G)$. Clearly, a GRR is necessarily a normal Cayley graph, and~\eqref{eq35} shows that a necessary condition for $\Cay(G,S)$ to be a GRR is $\Aut(G,S)=1$. In many circumstances this condition is also sufficient (see for instance~\cite{FLWX2002,Godsil1981,Godsil1983,LS2000,Spiga2018}), which prompts the following problem.

\begin{problem}\label{prob1}
\cite[Problem~A]{FLWX2002} Determine groups $G$ and Cayley graphs $\Cay(G,S)$ for which $\Cay(G,S)$ is a GRR of $G$ if and only if $\Aut(G,S)=1$.
\end{problem}

Partial results on Problem~\ref{prob1} have been achieved in the literature for $p$-groups $G$ (see~\cite{Godsil1981,LS2000,ZF2016}) and for cubic Cayley graphs $\Cay(G,S)$ on nonabelian simple groups $G$ (see~\cite{FLWX2002,Godsil1983}). Note by~\eqref{eq35} that Problem~\ref{prob1} is equivalent to determining groups $G$ and Cayley graphs $\Cay(G,S)$ for which $\Cay(G,S)$ is normal whenever $\Aut(G,S)=1$. We are then led to the problem as follows, which is a counterpart of Problem~\ref{prob1}.

\begin{problem}\label{prob2}
Determine groups $G$ and Cayley graphs $\Cay(G,S)$ for which $\Cay(G,S)$ is nonnormal and $\Aut(G,S)=1$.
\end{problem}

Compared with Problem~\ref{prob1}, cases in the solution of Problem~\ref{prob2} are expected to be more rare, and the Cayley graphs $\Gamma_m$ on alternating groups $\A_{2^{m+6}-1}$ as in Theorem~\ref{thm1} give examples of such cases.

\subsection{Nonnormal Cayley graphs on nonabelian simple groups}

In 2002, Fang, Praeger and Wang published a paper~\cite{FPW2002} describing the automorphism groups of possible nonnormal Cayley graphs on nonabelian simple groups. Since then, nonnormal Cayley graphs on nonabelian simple groups have attracted considerable attention, especially for those of small valency. In~\cite{XFWX2005,XFWX2007} it was proved that the only connected arc-transitive cubic nonnormal Cayley graphs on nonabelian simple groups are two Cayley graphs on $\A_{47}$ up to isomorphism, and their full automorphism groups are both isomorphic to $\A_{48}$. For connected cubic nonnormal Cayley graphs $\Cay(G,S)$ that are not arc-transitive, a list of candidates for the nonabelian simple group $G$ was obtained in~\cite{FLWX2002,ZF2010}. However, not many examples are known for connected cubic nonnormal Cayley graphs on nonabelian simple groups. In a survey paper~\cite{FLX2008} on normality of Cayley graphs, Feng, Lu and Xu asked the following question.

\begin{question}\label{que2}
\cite[Problem~5.9]{FLX2008} Are there infinitely many connected nonnormal Cayley graphs of valency $3$ or $4$ on nonabelian simple groups?
\end{question}

In 2012, Wang and Feng~\cite{WF2012} gave a positive answer to Question~\ref{que2} in the case of valency $4$ by constructing infinite families of connected tetravalent nonnormal Cayley graphs on alternating groups. Their examples are half-edge-transitive, where a \emph{half-edge-transitive} graph is by definition a graph whose automorphism group has exactly two orbits on the edge set. They also pointed out that many tetravalent nonnormal Cayley graphs turn out to be half-edge-transitive; for example, all the tetravalent nonnormal Cayley graphs on $\A_5$ and $\A_6$ are half-edge-transitive (see~\cite[Theorem~3.1]{XX2004} and~\cite[Theorem~5.1]{WF2012}).

Recently, the authors in~\cite{CXZ2018} constructed an infinite family of connected cubic nonnormal Cayley graphs on alternating groups, which completely answers Question~\ref{que2} in the affirmative. The construction in~\cite{CXZ2018} together with~\cite{Spiga2016} inspires the construction in the present paper (see Section~\ref{sec2}). This results in the graphs $\Gamma_m$ as in Theorem~\ref{thm1}, which form an infinite family of connected tetravalent nonnormal Cayley graphs on nonabelian simple groups that are edge-transitive.

\subsection{Tetravalent edge-transitive Cayley graphs on nonabelian simple groups}

As mentioned in the previous subsection, it is believed to be rare for a connected tetravalent nonnormal Cayley graph on a nonabelian simple group to be edge-transitive. In fact, the possibilities for such a nonabelian simple group are very restricted, as shown in~\cite[Theorem~1.1(2)]{FLX2004}: There are eight infinite families listed below and some sporadic cases in~\cite[Table~1]{FLX2004}.
\begin{enumerate}[(i)]
\item $\A_{2^n-1}$ for integer $n\geqslant3$;
\item $\PSL_n(2^e)$ for integers $n\geqslant4$ and $e\geqslant1$;
\item $\PSU_n(2^e)$ for integers $n\geqslant4$ and $e\geqslant1$;
\item $\PSp_{2n}(2^e)$ for integers $n\geqslant3$ and $e\geqslant1$;
\item $\E_6(2^e)$ for integer $e\geqslant1$;
\item $\E_7(2^e)$ for integer $e\geqslant1$;
\item $\,^2\E_6(2^e)$ for integer $e\geqslant1$;
\item $\,^2\G_2(2^e)$ for integer $e\geqslant1$.
\end{enumerate}
Nevertheless, it has been unknown whether the groups listed in~\cite[Theorem~1.1(2)]{FLX2004} indeed give examples of connected tetravalent edge-transitive nonnormal Cayley graphs. Thus we pose the following problem.

\begin{problem}\label{prob3}
Determine the groups $G$ in~\cite[Theorem~1.1(2)]{FLX2004} for which there exists a connected tetravalent edge-transitive nonnormal Cayley graph on $G$.
\end{problem}

Our Theorem~\ref{thm1} gives a solution of Problem~\ref{prob3} for the groups~(i) with $n\geqslant7$.

\subsection{Normal quotient analysis for tetravalent half-arc-transitive graphs}

Let $\Gamma$ be a tetravalent graph admitting a half-arc-transitive group $G$. For a normal subgroup $N$ of $G$, the \emph{normal quotient graph} $\Gamma_N$ is defined as follows: The vertex set of $\Gamma_N$ is the set of $N$-orbits on the vertex set of $\Gamma$, and a pair $\{B,C\}$ of distinct $N$-orbits forms an edge of $\Gamma_N$ if and only if there exists an edge $\{u,v\}$ of $\Gamma$ with $u\in B$ and $v\in C$. The normal quotient graph has been utilized to investigate various problems on graph symmetries (for instance the systematic study of $s$-arc-transitive graphs initiated by Praeger~\cite{Praeger1993}), and was recently proposed by Al-bar, Al-kenani, Muthana, Praeger and Spiga~\cite{AAMPS2016} as a new framework for the study of tetravalent graphs admitting a half-arc-transitive group.

Let $\mathcal{OG}(4)$ denote the family of pairs $(\Gamma,G)$ for which $\Gamma$ is a connected tetravalent graph admitting a half-arc-transitive group $G$. A pair $(\Gamma,G)$ in $\mathcal{OG}(4)$ is said to be \emph{basic} if $\Gamma_N$ has valency at most $2$ for each nontrivial normal subgroup $N$ of $G$. The framework proposed in~\cite{AAMPS2016} is to develop first a theory to describe the basic pairs in $\mathcal{OG}(4)$, and then a theory to describe the pairs $(\Gamma,G)\in\mathcal{OG}(4)$ for a basic pair $(\Gamma_N,\overline{G})\in\mathcal{OG}(4)$, where $N$ is a normal subgroup of $G$ and $\overline{G}$ is the permutation group induced by $G$ on the vertex set of $\Gamma_N$; see~\cite{AAMP2017i,AAMP2017ii} for progress under this framework. A permutation group is said to be \emph{quasiprimitive} if each of its nontrivial normal subgroups is transitive. As proposed in~\cite{AAMPS2016}, the first step to describe the basic pairs $(\Gamma,G)\in\mathcal{OG}(4)$ is:

\begin{problem}\label{prob4}
Describe the pairs $(\Gamma,G)\in\mathcal{OG}(4)$ with $G$ quasiprimitive on the vertex set of $\Gamma$.
\end{problem}

Since Theorem~\ref{thm1} implies that $\Aut(\Gamma_m)$ is simple and thus quasiprimitive on the vertex set of $\Gamma_m$, the graphs $\Gamma_m$ give rise to pairs $(\Gamma_m,\Aut(\Gamma_m))$ as examples in Problem~\ref{prob4}.

\subsection{$s$-Arc-transitive digraphs that are not self-reverse}

In general, a digraph is not necessarily self-reverse. However, constructing highly symmetric digraphs that are not self-reverse is nontrivial. In~\cite{Delorme2013} Delorme asked the following question.

\begin{question}\label{que3}
Do finite digraphs exist that are simultaneously vertex-transitive, $2$-arc-transitive and not self-reverse?
\end{question}

The answer to Question~\ref{que3} is affirmative, as shown by Conder, Poto\v{c}nik and \v{S}parl~\cite{CPS2015}. Then it is natural to ask the question in a stronger version by replacing $2$-arc-transitivity with $s$-arc-transitivity for larger values of $s$. In this fashion, our Theorem~\ref{thm1}(e) shows that the answer is always affirmative no matter how large $s$ is.

The digraphs $\calD_1(\Gamma_m)$ and $\calD_2(\Gamma_m)$ both admit $\Aut(\Gamma_m)\cong\A_{2^{m+6}}$ as a group of automorphisms, which is quasiprimitive on the vertex set. A systematic study of $s$-arc-transitive digraphs admitting a quasiprimitive group of automorphisms was initiated in~\cite{GX2018}, and examples of such digraphs for arbitrary large $s$ were first constructed in~\cite{CLP1995}.

\section{Construction of $\Gamma_m$}\label{sec2}

In this paper, as is the usual convention, the product over an empty set is $1$. For a permutation $\sigma$ of a set $\Omega$, denote the set of fixed points of $\sigma$ by $\Fix(\sigma)$.

Now we fix some notation that will be used throughout this paper. The key notation to be introduced in this section is listed in the table below.

\vskip0.1in
\begin{tabular}{ll}
\hline
$m$: & a positive integer\\
$E,e_1,\dots,e_m$: & $E=\langle e_1\rangle\times\dots\times\langle e_m\rangle\cong\C_2^m$\\
$H,a,b,c,d$: & $H=\langle a,b\mid a^4=b^2=(ab)^2=1\rangle\times\langle c,d\mid c^4=d^2=(cd)^2=1\rangle\times E$\\
$f$: & $ab\prod_{i=1}^{\lceil m/2\rceil}e_{2i-1}$\\
$K$: & $\langle a^2,b,c,d,E\rangle$\\
$R$: & the action of $H$ by right multiplication on itself\\
$x$: & the automorphism of $H$ defined by~\eqref{eq40}\\
$y$: & the permutation of $H$ defined by~\eqref{eq41}\\
$z$: & $R(f)yR(fcde_m^m)$\\
$\rho$: & the action of $\Alt(H)$ by right multiplication on $[\Alt(H){:}R(H)]$\\
$\Gamma_m$: & $\Cos(\Alt(H),R(H),R(H)\{xy,(xy)^{-1}\}R(H))$\\
\hline
\end{tabular}
\vskip0.1in

Let $m\geqslant1$ be an integer and let
\[
H=\langle a,b\mid a^4=b^2=(ab)^2=1\rangle\times\langle c,d\mid c^4=d^2=(cd)^2=1\rangle\times E,
\]
where $E=\langle e_1\rangle\times\dots\times\langle e_m\rangle\cong\C_2^m$. Clearly, $H\cong\D_8^2\times\C_2^m$, so that $|H|=2^{m+6}$. Let
\[
f=ab\prod_{i=1}^{\lceil m/2\rceil}e_{2i-1}\quad\text{and}\quad K=\langle a^2,b,c,d,E\rangle.
\]
Then $H=K\sqcup aK=K\sqcup fK$, where $\sqcup$ denotes the union of two disjoint sets. For convenience, put $e_i=1$ for $i\leqslant0$. In what follows we will define some permutations of $H$. Before giving the definition, we remind the reader that for $h\in H$ and $\sigma\in\Sym(H)$ the notation $h^\sigma$ exclusively means the image of $h$ under $\sigma$ in this paper.

Define $x\in\Aut(H)$ by letting
\begin{equation}\label{eq40}
a^x=c,\quad b^x=d,\quad c^x=a^3,\quad d^x=ab,\quad e_{2i-1}^x=e_{2i-1},\quad e_{2i}^x=a^2e_{2i-1}e_{2i}
\end{equation}
for $i=1,\dots,\lfloor m/2\rfloor$ and letting $e_m^x=a^2e_m$ in addition if $m$ is odd. Note that $x$ is indeed an automorphism of $H$ as the images of the generators of $H$ under $x$ are generators satisfying the defining relations of $H$. Define $\varphi\in\Aut(K)$ by letting
\[
(a^2)^\varphi=b,\ b^\varphi=a^2,\ c^\varphi=c,\ d^\varphi=d,\ e_{2i-1}^\varphi=e_{2i-3}e_{2i-2}e_{2i},\ e_{2i}^\varphi=e_{2i-3}e_{2i-2}e_{2i-1}
\]
for $i=1,\dots,\lfloor m/2\rfloor$ and letting $e_m^\varphi=e_m$ in addition if $m$ is odd. Note also that $\varphi$ is indeed an automorphism of $K$ as the images of the generators of $K$ under $\varphi$ are generators satisfying the defining relations of $K$. Let $y$ be the permutation of $H$ such that
\begin{equation}\label{eq41}
k^y=k^\varphi\quad\text{and}\quad(fk)^y=fcdk^\varphi e_m^m
\end{equation}
for each $k\in K$.

\begin{lemma}\label{lem6}
The permutations $x$ and $y$ have order $|x|=4$ and $|y|=2$.
\end{lemma}

\begin{proof}
It is direct to verify that $x^4$ fixes each of the generators $a,b,c,d,e_1,\dots,e_m$ of $H$, which implies $x^4=1$. Then since $a^{x^2}=(a^x)^x=c^x=a^3\neq a$, we know that $x^2\neq1$ and so $|x|=4$.

Since $\varphi^2$ fixes each of the generators $a^2,b,c,d,e_1,\dots,e_m$ of $K$ and $(a^2)^y=b\neq a^2$, it follows that $\varphi^2=1$ and $y\neq1$. Moreover, for each $k\in K$, as $k^\varphi\in K$, we have
\[
k^{y^2}=(k^\varphi)^y=(k^\varphi)^\varphi=k^{\varphi^2}=k
\]
and
\[
(fk)^{y^2}=(fcdk^\varphi e_m^m)^y=fcd(cdk^\varphi e_m^m)^\varphi e_m^m=fcd(cdk^{\varphi^2}e_m^m)e_m^m=fk^{\varphi^2}=fk.
\]
This shows that $y^2=1$, which together with $y\neq1$ implies $|y|=2$.
\end{proof}

Let $R$ be the action of $H$ by right multiplication on itself. Then $R(H)$ is a regular subgroup of $\Sym(H)$.

\begin{lemma}\label{lem3}
The permutations $x$ and $y$ and the group $R(H)$ are all in $\Alt(H)$.
\end{lemma}

\begin{proof}
According to Lemma~\ref{lem6}, $x$ has order $4$. For $i\in\{1,2,4\}$ let $n_i$ be the number of $i$-cycles in the cycle decomposition of $x$. Then $n_1+2n_2+4n_4=|H|=2^{m+6}$, $|\Fix(x)|=n_1$, $|\Fix(x^2)|=n_1+2n_2$, and $x\in\Alt(H)$ if and only if $n_2+n_4$ is even. Since $x\in\Aut(H)$, both $\Fix(x)$ and $\Fix(x^2)$ are subgroups of $H$. Note that
\[
\Fix(x)\geqslant
\begin{cases}
\langle ace_1\rangle\cong\C_4\quad&\text{if $m=1$}\\
\langle a^2c^2,e_1\rangle\cong\C_2^2\quad&\text{if $m\geqslant2$}
\end{cases}
\]
and
\[
\Fix(x^2)\geqslant
\begin{cases}
\langle a^2,ace_1\rangle\cong\C_2\times\C_4\quad&\text{if $m=1$}\\
\langle a^2,c^2,e_1\rangle\cong\C_2^3\quad&\text{if $m\geqslant2$.}
\end{cases}
\]
We conclude that $|\Fix(x)|$ and $|\Fix(x^2)|$ are divisible by $4$ and $8$, respectively. Hence
\[
n_2+n_4=\frac{n_1+2n_2+4n_4}{4}-\frac{n_1}{2}+\frac{n_1+2n_2}{4}=2^{m+4}-\frac{|\Fix(x)|}{2}+\frac{|\Fix(x^2)|}{4}
\]
is even, and so $x\in\Alt(H)$.

Let $\sigma$ be the bijection from $K$ to $fK$ sending $k$ to $fk$ for each $k\in K$, and $\tau$ be the permutation of $H$ such that $k^\tau=k$ and $(fk)^\tau=fcdke_m^m$ for each $k\in K$. Then
\[
(fk)^\tau\neq fk\quad\text{and}\quad(fk)^{\tau^2}=(fcdke_m^m)^\tau=fcd(cdke_m^m)e_m^m=fk
\]
for each $k\in K$. This implies that $\tau$ is a product of $|K|/2$ transpositions, and so $\tau\in\Alt(H)$. From the definition of $y$ we see that the following diagram commutes:
\[
\xymatrix{
K\ar[r]^{y\tau}\ar[d]^{\sigma}&K\ar[d]^{\sigma}\\
fK\ar[r]^{y\tau}&fK
}
\]
Hence $(y\tau)|_{fK}$ has the same cycle decomposition as $(y\tau)|_K$. As a consequence, $y\tau\in\Alt(H)$, which in conjunction with $\tau\in\Alt(H)$ yields $y\in\Alt(H)$.

Finally, since $H$ is a non-cyclic $2$-group, we derive that $R(H)\leqslant\Alt(H)$. This completes the proof.
\end{proof}

Recall the standard construction of the \emph{coset graph} $\Cos(A,B,S)$ for a finite group $A$ with a subgroup $B$ and an inverse-closed  subset $S$ such that $B\cap S=\emptyset$ and $S$ is a union of double cosets of $B$ in $A$. Such a graph has vertex set $[A{:}B]$, the set of right cosets of $B$ in $A$, and edge set $\{\{Bt,Bst\}\mid t\in A,\ s\in S\}$ (the condition $B\cap S=\emptyset$ ensures that there is no self-loop). It is easy to see that $A$ acts by right multiplication on $[A{:}B]$ as a group of automorphisms of $\Cos(A,B,S)$, and  $\Cos(A,B,S)$ has valency $|S|/|B|$. Moreover, $\Cos(A,B,S)$ is connected if and only if $A=\langle B,S\rangle$. Now we give the construction of $\Gamma_m$. Note from Lemma~\ref{lem3} that $R(H)$ and $xy$ are in $\Alt(H)$.

\begin{construction}\label{con1}
For each integer $m\geqslant1$, let
\[
\Gamma_m=\Cos(\Alt(H),R(H),R(H)\{xy,(xy)^{-1}\}R(H)).
\]
\end{construction}

As in the general construction of coset graphs, the right multiplication gives a group homomorphism $\rho$ from $\Alt(H)$ to $\Aut(\Gamma_m)$. In fact, $\rho$ is an embedding since $R(H)$ is core-free in $\Alt(H)$.

Throughout the paper, the permutation
\[
z:=R(f)yR(fcde_m^m)
\]
will play an important role. Observe that
\[
1^z=1^{R(f)yR(fcde_m^m)}=f^{yR(fcde_m^m)}=(fcde_m^m)^{R(fcde_m^m)}=(fcde_m^m)(fcde_m^m)=1.
\]
Let $\Alt(H)_1$ be the subgroup of $\Alt(H)$ stabilizing $1\in H$. Then $\Alt(H)_1\cong\A_{2^{m+6}-1}$, and $x,y,z\in\Alt(H)_1$ as they all fix $1$. We close this section with lemmas on some elements of $\langle x,y,z\rangle$, which will be used repeatedly (and sometimes implicitly) in subsequent sections. The lemmas are stated below without proof, as it is tedious calculation and can be found in~\cite{Xia}.

\begin{lemma}
The automorphism $x^{-1}$ of $H$ satisfies
\[
a^{x^{-1}}=c^3,\quad b^{x^{-1}}=cd,\quad c^{x^{-1}}=a,\quad d^{x^{-1}}=b,\quad e_{2i-1}^{x^{-1}}=e_{2i-1},\quad e_{2i}^{x^{-1}}=c^2e_{2i-1}e_{2i}
\]
for $i=1,\dots,\lfloor m/2\rfloor$ and $(e_m^m)^{x^{-1}}=c^{2m}e_m^m$.
\end{lemma}

\begin{lemma}\label{lem23}
For each $k\in K$, $(ak)^y=a^3bcdk^ye_{2\lfloor m/2\rfloor-1}e_{2\lfloor m/2\rfloor}e_m^m$.
\end{lemma}

\begin{lemma}\label{lem12}
The permutation $z$ is an involution in $\Alt(H)$ such that for each $g\in\langle c,d,E\rangle$,
\begin{align}\label{eq2}
g^z=cdg^ycd,\quad&(bg)^z=bcdg^ycd,\\
(a^2g)^z=a^2bcdg^ycd,\quad&(a^2bg)^z=a^2cdg^ycd,\\
(fg)^z=fg^ycde_m^m,\quad&(fbg)^z=fbg^ycde_m^m,\\
(fa^2g)^z=fa^2bg^ycde_m^m,\quad&(fa^2bg)^z=fa^2g^ycde_m^m.\label{eq3}
\end{align}
\end{lemma}

\begin{lemma}\label{lem9}
The following hold:
\begin{enumerate}[{\rm (a)}]
\item for each $g\in\langle c,d,E\rangle$,
\begin{align*}
g^{yz}=cdgcd,\quad&(bg)^{yz}=a^2bcdgcd,\\
(a^2g)^{yz}=bcdgcd,\quad&(a^2bg)^{yz}=a^2cdgcd,\\
(fg)^{yz}=fcdgcd,\quad&(fbg)^{yz}=fa^2bcdgcd,\\
(fa^2g)^{yz}=fbcdgcd,\quad&(fa^2bg)^{yz}=fa^2cdgcd;
\end{align*}
\item for each $g\in\langle c,d,E\rangle$,
\begin{align*}
(ag)^{yz}=a^3cdgcd,\quad&(abg)^{yz}=abcdgcd,\\
(a^3g)^{yz}=a^3bcdgcd,\quad&(a^3bg)^{yz}=acdgcd;
\end{align*}
\item $|yz|=6$;
\item $\Fix(yz)=\langle ab,c^2,cd,E\rangle$;
\item $\Fix((yz)^2)=\langle ab,c,d,E\rangle$;
\item $\Fix((yz)^3)=\langle a,b,c^2,cd,E\rangle$.
\end{enumerate}
\end{lemma}

\begin{lemma}\label{lem16}
The following hold:
\begin{enumerate}[{\rm (a)}]
\item if $m$ is even then $|\Fix(xyxz)|=3$;
\item if $m$ is odd then $|\Fix((xyxz)^2)|=3$.
\end{enumerate}
\end{lemma}

\begin{lemma}\label{lem2}
The following hold:
\begin{enumerate}[{\rm (a)}]
\item $|\Fix(yz)\cap\Fix(yz)^{xyzx^{-1}}|=2^{m+3}$;
\item $|\Fix(yz)^{x^{-1}}\cap\Fix(yz)^{x^{-1}yz}|=2^{m+1}$;
\item $|\Fix(yz)^{x^{-1}}\cap\Fix(yz)^{x^{-1}zy}|=2^{m+1}$;
\item if $m\equiv0\pmod{4}$, then
\[
|\Fix(yz)^{xz}\cap\Fix(xyxz)^{(xy)^{-1}(xz)^{-1}}|=1\quad\text{and}\quad|\Fix(zy)^{xy}\cap\Fix(xzxy)^{(xz)^{-1}(xy)^{-1}}|=2;
\]
\item if $m\equiv1\pmod{4}$, then
\[
|\Fix(yz)\cap\Fix((xyxz)^2)^{(xy)^2(xz)^2}|=1\quad\text{and}\quad|\Fix(zy)\cap\Fix((xzxy)^2)^{(xz)^2(xy)^2}|=2;
\]
\item if $m\equiv2\pmod{4}$, then
\[
|\Fix(yz)^{xz}\cap\Fix(xyxz)^{(xy)^{-1}(xz)^{-1}}|=2\quad\text{and}\quad|\Fix(zy)^{xy}\cap\Fix(xzxy)^{(xz)^{-1}(xy)^{-1}}|=1;
\]
\item if $m\equiv3\pmod{4}$, then
\[
|\Fix(yz)\cap\Fix((xyxz)^2)^{(xy)^2(xz)^2}|=2\quad\text{and}\quad|\Fix(zy)\cap\Fix((xzxy)^2)^{(xz)^2(xy)^2}|=1.
\]
\end{enumerate}
\end{lemma}

\begin{lemma}\label{lem21}
The set $\{x(yz)^ix^{-1},x(yz)^iy\mid i\in\mathbb{Z}\}\cap\{(yz)^j,(yz)^jyx^{-1}\mid j\in\mathbb{Z}\}$ is equal to $\{1\}$.
\end{lemma}

\section{Connectivity of $\Gamma_m$}\label{sec3}

The aim of this section is to prove that $\Gamma_m$ is connected. Let $V=\langle e_1\rangle\times\dots\times\langle e_{2\lfloor m/2\rfloor}\rangle$, and denote by $\mu$ the projection of $H$ to $V$. We need the following technical lemmas, whose proofs can be found in~\cite{Xia}.

\begin{lemma}\label{lem20}
For each $g\in cdE$, the projection $\mu$ maps $(abg)^{\langle xy,xz\rangle}\cap abcdE$ onto $V$.
\end{lemma}

\begin{lemma}\label{lem19}
The set $(abcdE)^{\langle xy,xz\rangle}$ contains $H\setminus\langle c^2,E\rangle$.
\end{lemma}

\begin{lemma}\label{lem13}
There is no nonempty subset of $\langle c^2,E\rangle\setminus\{1\}$ stabilized by $\langle xy\rangle$.
\end{lemma}

\begin{lemma}\label{lem15}
Let $B$ be a subgroup of $H$ stabilized by $\langle xy,xz\rangle$. Then $B=1$ or $H$.
\end{lemma}

Now we prove the connectivity of $\Gamma_m$.

\begin{proposition}\label{prop1}
The graph $\Gamma_m$ is connected.
\end{proposition}

\begin{proof}
Let $G=\langle R(H),xy\rangle$. Then Lemma~\ref{lem3} implies that $G\leqslant\Alt(H)$. Suppose for a contradiction that $\Gamma_m$ is disconnected. Then as
\[
\Gamma_m=\Cos(\Alt(H),R(H),R(H)\{xy,(xy)^{-1}\}R(H)),
\]
it follows that $G<\Alt(H)$. Note that $xy$ stabilizes $H\setminus\{1\}$. For each $g\in\langle c^2,E\rangle\setminus\{1\}$, we derive from Lemma~\ref{lem13} that $g^{\langle xy\rangle}\nsubseteq\langle c^2,E\rangle$. Accordingly, $g^{\langle xy,xz\rangle}\cap(H\setminus\langle c^2,E\rangle)\neq\emptyset$. This implies that $g\in(H\setminus\langle c^2,E\rangle)^{\langle xy,xz\rangle}$, which in conjunction with Lemma~\ref{lem19} yields $g\in(abcdE)^{\langle xy,xz\rangle}$. Since $g$ is an arbitrary element of $\langle c^2,E\rangle\setminus\{1\}$, it follows that $\langle c^2,E\rangle\setminus\{1\}\subseteq(abcdE)^{\langle xy,xz\rangle}$. Combining this with Lemma~\ref{lem19} we obtain
\begin{equation}\label{eq26}
H\setminus\{1\}\subseteq(abcdE)^{\langle xy,xz\rangle}.
\end{equation}

Suppose that $G$ has an imprimitive block system $\calB$ on $H$. Let $B$ be the block in $\calB$ containing $1$, where $1<|B|<|H|$. Then for each $g\in B$, as $R(g)$ maps $1$ to $g$, the permutation $R(g)$ stabilizes $B$, that is, $Bg=B$. This shows that $B$ is a subgroup of $H$. Since $x\in\Aut(H)$, we have $R(H)x=xR(H)$ and hence
\[
xz=xR(f)yR(fcde_m^m)\in xR(H)yR(H)=R(H)xyR(H)\subseteq G.
\]
As $xy$ and $xz$ both fix $1\in B$, we derive that $B$ is stabilized by $\langle xy,xz\rangle$, contradicting Lemma~\ref{lem15}. Therefore, $G$ is primitive.

From Lemma~\ref{lem9}(f) we see that the fixed point ratio $|\Fix((yz)^3)|/|H|$ of $(yz)^3$ is $|\langle a,b,c^2,cd,E\rangle|/|H|=1/2$. Then since $|H|$ is a power of $2$, we conclude from~\cite[Theorem~1]{GM1998} that one of the following two cases occurs:
\begin{enumerate}[{\rm (i)}]
\item $G$ is an affine group over $\mathbb{F}_2$.
\item $G$ has socle $\A_n^\ell$ for some integers $n\geqslant5$ and $\ell\geqslant1$ such that $G\leqslant\Sy_n\wr\Sy_\ell$, where $\Sy_n$ acts on the set of $k$-subsets of an $n$-set for some $k\leqslant n/4$ and the wreath product is in product action.
\end{enumerate}
If~(i) occurs, then every element of the stabilizer $G_1$ of $1$ is a linear transformation over $\mathbb{F}_2$ and hence has the number of fixed points a power of $2$. However, it follows from Lemma~\ref{lem16} that $G_1$ has an element with exactly three fixed points. Thus~(i) does not occur. Now we have~(ii). In particular, $|H|=\binom{n}{k}^\ell$, and so $\binom{n}{k}$ is a power of $2$. This implies that the subgroup of $\A_n$ stabilizing $\{1,\dots,k\}$ has index a power of $2$ in $\A_n$. We then conclude from~\cite[Theorem~1]{Guralnick1983} that $k=1$. Hence $|H|=n^\ell$. Then as $G<\Alt(H)$ and $G$ has socle $\A_n^\ell$, it follows that $\ell\geqslant2$. Consequently, $G$ is not a $2$-transitive group, and so $G_1$ has at least three orbits on $H$. In particular, $\langle xy,xz\rangle$ has at least three orbits on $H$, which implies that $\langle xy,xz\rangle$ is not transitive on $H\setminus\{1\}$.

Take any $h_1\in abcdE$. If $abcdE\subseteq h_1^{\langle xy,xz\rangle}$, then~\eqref{eq26} would imply that $H\setminus\{1\}\subseteq h_1^{\langle xy,xz\rangle}$ and thus ${\langle xy,xz\rangle}$ is transitive on $H\setminus\{1\}$, a contradiction. Hence $abcdE\nsubseteq h_1^{\langle xy,xz\rangle}$. Take any $h_2\in abcdE\setminus h_1^{\langle xy,xz\rangle}$. Then
\begin{equation}\label{eq27}
h_1^{\langle xy,xz\rangle}\cap h_2^{\langle xy,xz\rangle}=\emptyset.
\end{equation}
Note from Lemma~\ref{lem20} that $|h_i^{\langle xy,xz\rangle}\cap abcdE|\geqslant|V|=|abcdE|/\gcd(2,m-1)$ for $i=1,2$. We conclude from~\eqref{eq27} that $m$ is odd and
\[
|h_i^{\langle xy,xz\rangle}\cap abcdE|\geqslant|abcdE|/2\quad\text{for $i=1,2$}.
\]
Combining this with~\eqref{eq27} we derive that $abcdE\subseteq h_1^{\langle xy,xz\rangle}\cup h_2^{\langle xy,xz\rangle}$, which together with~\eqref{eq26} leads to $H\setminus\{1\}\subseteq h_1^{\langle xy,xz\rangle}\cup h_2^{\langle xy,xz\rangle}$. This implies that $\langle xy,xz\rangle$ has at most three orbits on $H$. Then as $G\leqslant\Sy_n\wr\Sy_\ell$, where $\Sy_n$ acts naturally on $n$ points and the wreath product is in product action, it follows that $\ell=2$ and so $|H|=n^2$ is a square. However, $|H|=|\D_8^2\times\C_2^m|=2^{m+6}$ with $m$ odd, a contradiction. The proof is thus complete.
\end{proof}

\section{$\Gamma_m$ as a Cayley graph}\label{sec4}

In this section we shall prove that $\Gamma_m$ is isomorphic to $\Cay(\Alt(H)_1,S)$, where $S=\{xy,(xy)^{-1},xz,(xz)^{-1}\}$, and $\Cay(\Alt(H)_1,S)$ is a nonnormal Cayley graph on $\Alt(H)_1\cong\A_{2^{m+6}-1}$ with $\Aut(\Alt(H)_1,S)=1$.

\begin{lemma}\label{lem1}
The following hold:
\begin{enumerate}[{\rm (a)}]
\item $R(H)xyR(H)=R(H)xy\sqcup R(H)xz$;
\item $R(H)(xy)^{-1}R(H)=R(H)(xy)^{-1}\sqcup R(H)(xz)^{-1}$;
\item $R(H)\{xy,(xy)^{-1}\}R(H)=R(H)xy\sqcup R(H)(xy)^{-1}\sqcup R(H)xz\sqcup R(H)(xz)^{-1}$.
\end{enumerate}
\end{lemma}

\begin{proof}
If $(xy)^{-1}R(H)xy\cap R(H)=R(H)$, then $R(H)$ is normal in $\langle R(H),xy\rangle$, which is impossible as Proposition~\ref{prop1} implies that $\langle R(H),xy\rangle=\Alt(H)$. Thus
\[
(xy)^{-1}R(H)xy\cap R(H)\neq R(H).
\]
Since $x\in\Aut(H)$ and $y|_K\in\Aut(K)$, we have $x^{-1}R(H)x=R(H)$ and $yR(K)y=R(K)$, which implies that
\begin{align*}
(xy)^{-1}R(H)xy\cap R(H)&=yx^{-1}R(H)xy\cap R(H)\\
&=yR(H)y\cap R(H)\geqslant yR(K)y\cap R(K)=R(K).
\end{align*}
As $R(K)$ has index $2$ in $R(H)$, we then deduce that $(xy)^{-1}R(H)xy\cap R(H)=R(K)$. In particular, $(xy)^{-1}R(H)xy$ has index $2$ in $R(H)$, whence
\[
\frac{|R(H)xyR(H)|}{|R(H)|}=\frac{|R(H)|}{|(xy)^{-1}R(H)xy\cap R(H)|}=2.
\]
Consequently,
\begin{equation}\label{eq20}
|R(H)(xy)^{-1}R(H)|=|R(H)xyR(H)|=2|R(H)|
\end{equation}
and thus
\begin{equation}\label{eq19}
|R(H)\{xy,(xy)^{-1}\}R(H)|\leqslant|R(H)xyR(H)|+|R(H)(xy)^{-1}R(H)|=4|R(H)|.
\end{equation}
Note from the definition of $z$ that
\[
xz\in xR(H)yR(H)=R(H)xyR(H).
\]
Hence $R(H)xz\subseteq R(H)xyR(H)$ and $R(H)(xz)^{-1}\subseteq R(H)(xy)^{-1}R(H)$. Moreover, from
\[
b^{xy}=d,\quad b^{(xy)^{-1}}=c^2,\quad b^{xz}=c^2d\quad\text{and}\quad b^{(xz)^{-1}}=cd
\]
we see that $xy$, $(xy)^{-1}$, $xz$ and $(xz)^{-1}$ are pairwise distinct. Then as
\[
xy,(xy)^{-1},xz,(xz)^{-1}\in\Alt(H)_1
\]
and $\Alt(H)_1$ forms a right transversal of $R(H)$ in $\Alt(H)$, it follows that $R(H)xy$, $R(H)(xy)^{-1}$, $R(H)xz$ and $R(H)(xz)^{-1}$ are pairwise disjoint. Therefore,
\[
R(H)xyR(H)\supseteq R(H)xy\sqcup R(H)xz,
\]
\[
R(H)(xy)^{-1}R(H)\supseteq R(H)(xy)^{-1}\sqcup R(H)(xz)^{-1}
\]
and
\[
R(H)\{xy,(xy)^{-1}\}R(H)\supseteq R(H)xy\sqcup R(H)(xy)^{-1}\sqcup R(H)xz\sqcup R(H)(xz)^{-1}.
\]
This combined with~\eqref{eq20} and~\eqref{eq19} yields the lemma.
\end{proof}

\begin{proposition}\label{prop2}
Let $S=\{xy,(xy)^{-1},xz,(xz)^{-1}\}$. Then $\Cay(\Alt(H)_1,S)$ is a tetravalent nonnormal Cayley graph isomorphic to $\Gamma_m$ by the mapping $g\mapsto R(H)g$.
\end{proposition}

\begin{proof}
From Lemma~\ref{lem1}(c) we see that $xy$, $(xy)^{-1}$, $xz$ and $(xz)^{-1}$ are pairwise distinct. Hence $|S|=4$ and so $\Cay(\Alt(H)_1,S)$ is tetravalent. Let $\sigma\colon g\mapsto R(H)g$ be the mapping from $\Alt(H)_1$ to the vertex set of $\Gamma_m$. Since $\Alt(H)_1$ forms a right transversal of $R(H)$ in $\Alt(H)$, the mapping $\sigma$ is bijective. Moreover, for $u$ and $v$ in $\Alt(H)_1$, $u$ is adjacent to $v$ in $\Cay(\Alt(H)_1,S)$ if and only if $vu^{-1}\in S$, which is equivalent to
\[
R(H)vu^{-1}\in\{R(H)xy,R(H)(xy)^{-1},R(H)xz,R(H)(xz)^{-1}\}.
\]
By Lemma~\ref{lem1}, this means that $u$ and $v$ is adjacent in $\Cay(\Alt(H)_1,S)$ if and only if
\[
R(H)vu^{-1}\subseteq R(H)\{xy,(xy)^{-1}\}R(H),
\]
or equivalently, $R(H)u$ is adjacent to $R(H)v$ in $\Gamma_m$. Therefore, $\sigma$ is a graph isomorphism from $\Cay(\Alt(H)_1,S)$ to $\Gamma_m$. It follows that
\[
\sigma\rho(\Alt(H)_1)\sigma^{-1}<\sigma\rho(\Alt(H))\sigma^{-1}\leqslant\Aut(\Cay(\Alt(H)_1,S)).
\]
Note that for each $g\in\Alt(H)_1$, the permutation $\sigma\rho(g)\sigma^{-1}$ of $\Alt(H)_1$ is precisely the right multiplication of $g$. Then since $\sigma\rho(\Alt(H)_1)\sigma^{-1}$ is not normal in $\sigma\rho(\Alt(H))\sigma^{-1}$, we conclude that $\Cay(\Alt(H)_1,S)$ is a nonnormal Cayley graph. This completes the proof.
\end{proof}

\begin{proposition}\label{prop3}
Let $S=\{xy,(xy)^{-1},xz,(xz)^{-1}\}$. Then $\Aut(\Alt(H)_1,S)=1$.
\end{proposition}

\begin{proof}
Since $\Gamma_m$ is connected as Proposition~\ref{prop1} states, we derive from Proposition~\ref{prop2} that $\Cay(\Alt(H)_1,S)$ is connected, which means $\langle xy,xz\rangle=\Alt(H)_1$. Suppose for a contradiction that $\Aut(\Alt(H)_1,S)\neq1$. Then there exists $1\neq\sigma\in\Sym(H)$ such that $1^\sigma=1$ and
\[
\{\sigma^{-1}xy\sigma,\sigma^{-1}(xy)^{-1}\sigma,\sigma^{-1}xz\sigma,\sigma^{-1}(xz)^{-1}\sigma\}=\{xy,(xy)^{-1},xz,(xz)^{-1}\}.
\]
In particular, $\sigma^{-1}xy\sigma\in\{xy,(xy)^{-1},xz,(xz)^{-1}\}$.

\underline{Case~1.} Assume $\sigma^{-1}xy\sigma=xy$. Then $\sigma^{-1}(xy)^{-1}\sigma=(xy)^{-1}$, and so $\sigma^{-1}xz\sigma=xz$ or $(xz)^{-1}$. If $\sigma^{-1}xz\sigma=xz$, then since $\langle xy,xz\rangle=\Alt(H)_1$, we conclude that $\sigma$ centralizes every element of $\Alt(H)_1$, contradicting $\sigma\neq1$. Thus $\sigma^{-1}xz\sigma=(xz)^{-1}$. It follows that
\[
|\Fix(yz)|=|\Fix(xy(xz)^{-1})|=|\Fix((\sigma^{-1}xy\sigma)(\sigma^{-1}xz\sigma))|=|\Fix(xyxz)|
\]
and
\[
|\Fix((yz)^2)|=|\Fix((xy(xz)^{-1})^2)|=|\Fix(((\sigma^{-1}xy\sigma)(\sigma^{-1}xz\sigma))^2)|=|\Fix((xyxz)^2)|.
\]
However, according to Lemmas~\ref{lem9} and~\ref{lem16}, these equations cannot hold simultaneously, a contradiction.

\underline{Case~2.} Assume $\sigma^{-1}xy\sigma=(xy)^{-1}$. Then $\sigma^{-1}(xy)^{-1}\sigma=xy$, and so $\sigma^{-1}xz\sigma=xz$ or $(xz)^{-1}$. If $\sigma^{-1}xz\sigma=xz$, then
\[
|\Fix(yz)|=|\Fix((xy)^{-1}xz)|=|\Fix((\sigma^{-1}xy\sigma)(\sigma^{-1}xz\sigma))|=|\Fix(xyxz)|
\]
and
\[
|\Fix((yz)^2)|=|\Fix(((xy)^{-1}xz)^2)|=|\Fix(((\sigma^{-1}xy\sigma)(\sigma^{-1}xz\sigma))^2)|=|\Fix((xyxz)^2)|,
\]
which cannot hold simultaneously by Lemmas~\ref{lem9} and~\ref{lem16}, a contradiction. Therefore, $\sigma^{-1}xz\sigma=(xz)^{-1}$, and so $\sigma^{-1}(xz)^{-1}\sigma=xz$. However, this implies that
\begin{align*}
&\left|\Fix(yz)^{x^{-1}}\cap\Fix(yz)^{x^{-1}yz}\right|\\
=&\left|\Fix(xy(xz)^{-1})\cap\Fix(xy(xz)^{-1})^{(xy)^{-1}xz}\right|\\
=&\left|\Fix((\sigma^{-1}(xy)^{-1}\sigma)(\sigma^{-1}xz\sigma))\cap
\Fix((\sigma^{-1}(xy)^{-1}\sigma)(\sigma^{-1}xz\sigma))^{(\sigma^{-1}xy\sigma)(\sigma^{-1}(xz)^{-1}\sigma)}\right|\\
=&\left|\Fix((xy)^{-1}xz)\cap\Fix((xy)^{-1}xz)^{xy(xz)^{-1}}\right|\\
=&\left|\Fix(yz)\cap\Fix(yz)^{xyzx^{-1}}\right|,
\end{align*}
contradicting Lemma~\ref{lem2}(a)(b).

\underline{Case~3.} Assume $\sigma^{-1}xy\sigma=xz$. Then $\sigma^{-1}(xy)^{-1}\sigma=(xz)^{-1}$, and so $\sigma^{-1}xz\sigma=xy$ or $(xy)^{-1}$. If $\sigma^{-1}xz\sigma=(xy)^{-1}$, then
\[
|\Fix(yz)|=|\Fix(xz(xy)^{-1})|=|\Fix((\sigma^{-1}xy\sigma)(\sigma^{-1}xz\sigma))|=|\Fix(xyxz)|
\]
and
\[
|\Fix((yz)^2)|=|\Fix((xz(xy)^{-1})^2)=|\Fix(((\sigma^{-1}xy\sigma)(\sigma^{-1}xz\sigma))^2)|=|\Fix((xyxz)^2)|,
\]
which cannot hold simultaneously by Lemmas~\ref{lem9} and~\ref{lem16}, a contradiction. Therefore, $\sigma^{-1}xz\sigma=xy$, which leads to
\begin{align*}
&\left|\Fix(yz)^{xz}\cap\Fix(xyxz)^{(xy)^{-1}(xz)^{-1}}\right|\\
=&\left|\Fix((xy)^{-1}xz)^{xz}\cap\Fix(xyxz)^{(xy)^{-1}(xz)^{-1}}\right|\\
=&\left|\Fix((\sigma^{-1}xz\sigma)^{-1}(\sigma^{-1}xy\sigma))^{\sigma^{-1}xy\sigma}
\cap\Fix((\sigma^{-1}xz\sigma)(\sigma^{-1}xy\sigma))^{(\sigma^{-1}xz\sigma)^{-1}(\sigma^{-1}xy\sigma)^{-1}}\right|\\
=&\left|\Fix((xz)^{-1}xy)^{xy}\cap\Fix(xzxy)^{(xz)^{-1}(xy)^{-1}}\right|\\
=&\left|\Fix(zy)^{xy}\cap\Fix(xzxy)^{(xz)^{-1}(xy)^{-1}}\right|
\end{align*}
and
\begin{align*}
&\left|\Fix(yz)\cap\Fix((xyxz)^2)^{(xy)^2(xz)^2}\right|\\
=&\left|\Fix((xy)^{-1}xz)\cap\Fix((xyxz)^2)^{(xy)^2(xz)^2}\right|\\
=&\left|\Fix((\sigma^{-1}xz\sigma)^{-1}(\sigma^{-1}xy\sigma))
\cap\Fix(((\sigma^{-1}xz\sigma)(\sigma^{-1}xy\sigma))^2)^{(\sigma^{-1}xz\sigma)^2(\sigma^{-1}xy\sigma)^2}\right|\\
=&\left|\Fix((xz)^{-1}xy)\cap\Fix((xzxy)^2)^{(xz)^2(xy)^2}\right|\\
=&\left|\Fix(zy)\cap\Fix((xzxy)^2)^{(xz)^2(xy)^2}\right|,
\end{align*}
contradicting Lemma~\ref{lem2}(d)--(g).

\underline{Case~4.} Assume $\sigma^{-1}xy\sigma=(xz)^{-1}$. Then $\sigma^{-1}(xy)^{-1}\sigma=xz$, and so $\sigma^{-1}xz\sigma=xy$ or $(xy)^{-1}$. If $\sigma^{-1}xz\sigma=xy$, then
\[
|\Fix(yz)|=|\Fix((xz)^{-1}xy)|=|\Fix((\sigma^{-1}xy\sigma)(\sigma^{-1}xz\sigma))|=|\Fix(xyxz)|
\]
and
\[
|\Fix((yz)^2)|=|\Fix(((xz)^{-1}xy)^2)|=|\Fix(((\sigma^{-1}xy\sigma)(\sigma^{-1}xz\sigma))^2)|=|\Fix((xyxz)^2)|,
\]
which cannot hold simultaneously by Lemmas~\ref{lem9} and~\ref{lem16}, a contradiction. Thus $\sigma^{-1}xz\sigma=(xy)^{-1}$, and so $\sigma^{-1}(xz)^{-1}\sigma=xy$. Since $(yz)^{-1}=zy$, we have $\Fix(yz)=\Fix(zy)$. It then follows that
\begin{align*}
&\left|\Fix(yz)^{x^{-1}}\cap\Fix(yz)^{x^{-1}zy}\right|\\
=&\left|\Fix(zy)^{x^{-1}}\cap\Fix(zy)^{x^{-1}zy}\right|\\
=&\left|\Fix(xz(xy)^{-1})\cap\Fix(xz(xy)^{-1})^{(xz)^{-1}xy}\right|\\
=&\left|\Fix((\sigma^{-1}(xy)^{-1}\sigma)(\sigma^{-1}xz\sigma))\cap
\Fix((\sigma^{-1}(xy)^{-1}\sigma)(\sigma^{-1}xz\sigma))^{(\sigma^{-1}xy\sigma)(\sigma^{-1}(xz)^{-1}\sigma)}\right|\\
=&\left|\Fix((xy)^{-1}xz)\cap\Fix((xy)^{-1}xz)^{xy(xz)^{-1}}\right|\\
=&\left|\Fix(yz)\cap\Fix(yz)^{xyzx^{-1}}\right|,
\end{align*}
contradicting Lemma~\ref{lem2}(a)(c). This completes the proof.
\end{proof}

\section{$\Gamma_m$ as a half-arc-transitive graph}\label{sec5}

In this section we determine $\Aut(\Gamma_m)$ and show that $\Gamma_m$ is a loosely attached tetravalent half-arc-transitive graph of radius $6$, and then prove the $(m+6)$-arc-transitivity of $\calD_1(\Gamma)$ and $\calD_2(\Gamma)$. This together with results from previous sections finally leads to a proof of Theorem~\ref{thm1} at the end of the section.

\begin{proposition}\label{prop4}
The automorphism group of $\Gamma_m$ is $\rho(\Alt(H))$, whose stabilizer of the vertex $R(H)$ has $\{R(H)xy,R(H)xz\}$ and $\{R(H)(xy)^{-1},R(H)(xz)^{-1}\}$ as the orbits on the neighborhood of $R(H)$ in $\Gamma_m$.
\end{proposition}

\begin{proof}
According to Proposition~\ref{prop1}, $\Gamma_m$ is connected. Let $X=\Aut(\Gamma_m)$, let $Y$ be the stabilizer in $X$ of the vertex $R(H)$, and let $G=\rho(\Alt(H)_1)$. Then $G\cong\A_{2^{m+6}-1}$ and $G$ is regular on the vertex set $[\Alt(H){:}R(H)]$ of $\Gamma_m$. In particular, $|G|=|X|/|Y|$. From Lemma~\ref{lem1} we see that
\[
R(H)xyR(H)=R(H)xzR(H)=R(H)xy\sqcup R(H)xz
\]
and
\[
R(H)(xy)^{-1}R(H)=R(H)(xz)^{-1}R(H)=R(H)(xy)^{-1}\sqcup R(H)(xz)^{-1},
\]
which implies that $\rho(R(H))$ has $\{R(H)xy,R(H)xz\}$ and $\{R(H)(xy)^{-1},R(H)(xz)^{-1}\}$ as the orbits on the neighborhood of $R(H)$ in $\Gamma_m$. As a consequence, $X\geqslant\rho(\Alt(H))$ is edge-transitive, and the induced permutation group $\overline{Y}$ of $Y$ on the neighborhood of $R(H)$ either is transitive or has two orbits of size $2$. If $|\overline{Y}|$ is divisible by $3$, then $\overline{Y}$ is $2$-transitive and it follows from a well-known result of Gardiner (see for instance~\cite[Lemma~2.3]{FLX2004}) that $|Y|$ divides $2^43^6$, contradicting $Y\geqslant\rho(R(H))\cong\D_8^2\times\C_2^m$. Thus $\overline{Y}$ is a $2$-group and so $Y$ is a $2$-group. Accordingly, $|X|/|G|=|Y|$ is a power of $2$. By Propositions~\ref{prop2} and~\ref{prop3}, $\Gamma_m$ is isomorphic to a Cayley graph on $\Alt(H)_1$ with connection set $S$ such that $\Aut(\Alt(H)_1,S)=1$, and the right multiplication action of $\Alt(H)_1$ on this Cayley graph is equivalent to that of $G=\rho(\Alt(H)_1)$ on $\Gamma_m$. Then we conclude from~\eqref{eq35} that $\Nor_X(G)=G$. Note that every nontrivial $G$-conjugacy class has size greater than $4$, the valency of $\Gamma_m$. We derive from~\cite[Theorem~1.1]{FPW2002} that one of the following two cases occurs:
\begin{enumerate}[(i)]
\item $\Soc(X)$ is a nonabelian simple group containing $G$ as a proper subgroup;
\item $X$ has a nontrivial normal subgroup $N$ intransitive on the vertex set of $\Gamma_m$ such that $\Soc(X/N)$ is a nonabelian simple group containing $GN/N\cong G$ and is transitive on the vertex set of the quotient graph of $\Gamma_m$ with respect to $N$.
\end{enumerate}
In what follows we will identify $\rho(\Alt(H))$ with $\Alt(H)$ by abuse of notation.

First assume that case~(i) occurs. Since $|\Soc(X)|/|G|$ divides $|X|/|G|$ and $|X|/|G|$ is a power of $2$, we see that $|\Soc(X)|/|G|$ is a power of $2$. Then it follows from~\cite[Theorem~1]{Guralnick1983} that $\Soc(X)=\A_{2^{m+6}}$, whence $X\cong\A_{2^{m+6}}$ or $\Sy_{2^{m+6}}$. If $X\cong\Sy_{2^{m+6}}$, then $\Nor_X(G)\cong\Sy_{2^{m+6}-1}$, contrary to the conclusion that $\Nor_X(G)=G$. Therefore, $X\cong\A_{2^{m+6}}$. As $X\geqslant\Alt(H)\cong\A_{2^{m+6}}$, this leads to $X=\Alt(H)$, and so $Y=\rho(R(H))$. Hence $Y$ has $\{R(H)xy,R(H)xz\}$ and $\{R(H)(xy)^{-1},R(H)(xz)^{-1}\}$ as orbits on the neighborhood of $R(H)$ in $\Gamma_m$, as $\rho(R(H))$ does.

Next assume that case~(ii) occurs. Then $G\cap N=1$, and $|N|=|GN|/|G|$ divides $|X|/|G|$, which shows that $N$ is a $2$-group. Consider the action $\theta$ of $\Alt(H)$ on $N$ by conjugation. Since $\Alt(H)$ is simple, either $\ker(\theta)=1$ or $\ker(\theta)=\Alt(H)$. If $\ker(\theta)=\Alt(H)$, then $\Alt(H)$ centralizes $N$ and hence $N\leqslant\Nor_X(G)=G$, contradicting $G\cap N=1$. Thus $\ker(\theta)=1$, and so $\Alt(H)\cong\A_{2^{m+6}}$ is isomorphic to an irreducible subgroup of $\Aut(N/\Phi(N))\cong\SL_n(2)$ for some positive integer $n$ with $2^n\leqslant|N|$, where $\Phi(N)$ is the Frattini subgroup of $N$. Then by~\cite[Proposition~5.3.7]{KL1990} we have $n\geqslant2^{m+6}-2$, which implies that $\log_2|N|\geqslant2^{m+6}-2$. Since $\Gamma_m$ is isomorphic to a Cayley graph on the simple group $\A_{2^{m+6}-1}$ and $X$ is edge-transitive, it follows that $\Gamma_m$ is non-bipartite and $X$ is either arc-transitive or half-arc-transitive. Then we conclude from~\cite[Theorem~1.1]{GP1994} and~\cite[Theorem~1.1]{AAMPS2016} that $N$ is semiregular on the vertex set of $\Gamma_m$, for otherwise the quotient graph of $\Gamma_m$ with respect to $N$ would have valency at most $2$ and could not admit a vertex-transitive action of the nonabelian simple group $\Soc(X/N)$. As a consequence, $|N|$ divides $|\Alt(H)|/|R(H)|=(2^{m+6}-1)!/2$ and hence
\[
\log_2|N|\leqslant\log_2\frac{(2^{m+6}-1)!}{2}=\sum_{i=1}^\infty\left\lfloor\frac{2^{m+6}-1}{2^i}\right\rfloor-1<\sum_{i=1}^\infty\frac{2^{m+6}-1}{2^i}-1=2^{m+6}-2.
\]
This contradicts the conclusion $\log_2|N|\geqslant2^{m+6}-2$, completing the proof.
\end{proof}

\begin{proposition}\label{prop5}
$\Gamma_m$ is a loosely attached tetravalent half-arc-transitive graph of radius~$6$, and $\Aut(\Gamma_m)$ has vertex stabilizer isomorphic to $\D_8^2\times\C_2^m$.
\end{proposition}

\begin{proof}
As $\Gamma_m=\Cos(\Alt(H),R(H),R(H)\{xy,(xy)^{-1}\}R(H))$, we deduce from Proposition~\ref{prop4} that $\Gamma_m$ is a tetravalent half-arc-transitive graph and $\Aut(\Gamma_m)$ has vertex stabilizer $\rho(R(H))$. Consequently, $\Aut(\Gamma_m)$ has vertex stabilizer isomorphic to $H\cong\D_8^2\times\C_2^m$. Let $\Sigma_m=\Cay(\Alt(H)_1,\{xy,(xy)^{-1},xz,(xz)^{-1}\}$ and let $G$ be the automorphism group of $\Sigma_m$. Combining Proposition~\ref{prop2} and Proposition~\ref{prop4} we derive that $\Gamma_m\cong\Sigma_m$, and $\{xy,xz\}$ and $\{(xy)^{-1},(xz)^{-1}\}$ are the orbits of $G_1$ on the neighborhood of $1$ in $\Sigma_m$. Then since $\Sigma_m$ is a Cayley graph on $\Alt(H)_1$, it follows that, for each $g\in\Alt(H)_1$, the stabilizer of $g$ in $G$ has $\{xyg,xzg\}$ and $\{(xy)^{-1}g,(xz)^{-1}g\}$ as the orbits on the neighborhood of $g$ in $\Sigma_m$. Hence the Cayley digraphs $\Cay(\Alt(H)_1,\{xy,xz\})$ and $\Cay(\Alt(H)_1,\{(xy)^{-1},(xz)^{-1}\})$ are of different orientations given by the half-arc-transitive graph $\Sigma_m$.

Let $C_1$ be the alternating cycle of the tetravalent half-arc-transitive graph $\Sigma_m$ containing the edge $\{1,xy\}$. Then $C_1$ consists of vertices
\[
1,xy,xy(xz)^{-1},xy(xz)^{-1}xy,\dots,(xy(xz)^{-1})^{r-1},(xy(xz)^{-1})^{r-1}xy,
\]
where $r$ is the smallest positive integer such that $(xy(xz)^{-1})^r=1$. Note that $(xy(xz)^{-1})^i=x(yz)^ix^{-1}$ for all integers $i\geqslant0$. We deduce from Lemma~\ref{lem9}(c) that $r=6$. Hence $\Sigma_m$ has attachment number $6$, and the alternating cycle $C_1$ has vertex set
\[
\{x(yz)^ix^{-1},x(yz)^iy\mid i=0,1,2,3,4,5\}.
\]
Similarly, the alternating cycle $C_2$ of $\Sigma_m$ containing the edge $\{1,(xy)^{-1}\}$ has vertex set
\[
\{(yz)^j,(yz)^jyx^{-1}\mid j=0,1,2,3,4,5\}.
\]
From Lemma~\ref{lem21} we see that $C_1$ and $C_2$ has the only common vertex $1$. Thus $\Sigma_m$ has attachment number $1$. Then as $\Gamma_m\cong\Sigma_m$, we conclude that $\Gamma_m$ has radius $6$ and attachment number $1$, which completes the proof.
\end{proof}

Now we know that $\Gamma_m$ is a half-arc-transitive graph. Thus it gives rise to two digraphs $\calD_1(\Gamma_m)$ and $\calD_2(\Gamma_m)$ of opposite orientations. In the next proposition we prove the properties of $\calD_1(\Gamma_m)$ and $\calD_2(\Gamma_m)$ claimed in Theorem~\ref{thm1}(e). The proof of the $(m+6)$-arc-transitivity follows from the argument in~\cite[p.~190--191]{SV2016} and is given by an anonymous referee.

\begin{proposition}\label{prop6}
The digraphs $\calD_1(\Gamma_m)$ and $\calD_2(\Gamma_m)$ are $(m+6)$-arc-transitive and not self-reverse.
\end{proposition}

\begin{proof}
The digraphs $\calD_1(\Gamma_m)$ and $\calD_2(\Gamma_m)$ have the same vertex set as $\Gamma_m$, and have the arc sets the two orbits of $\Aut(\Gamma_m)$ on the arc set of $\Gamma_m$. The half-arc-transitivity of $\Gamma_m$ implies that there is no automorphism of $\Gamma_m$ interchanging these two orbits. Hence $\calD_i(\Gamma_m)$ has the same automorphism group as $\Gamma_m$, and there is no isomorphism between $\calD_1(\Gamma_m)$ and $\calD_2(\Gamma_m)$. Since $\calD_1(\Gamma_m)$ and $\calD_2(\Gamma_m)$ are the reverse of each other, it follows that none of $\calD_1(\Gamma_m)$ and $\calD_2(\Gamma_m)$ is self-reverse.

Let $G=\Aut(\Gamma_m)$. Then $G$ is the automorphism group of $\calD_i(\Gamma_m)$ for $i=1,2$. Since the digraph $\calD_i(\Gamma_m)$ is vertex-transitive with the underlying graph $\Gamma_m$ connected, we see that $\calD_i(\Gamma_m)$ is strongly connected. Let $t$ be the largest integer such that $G$ is $t$-arc-transitive on $\calD_i(\Gamma_m)$. Take a $t$-arc $(v_0,v_1,\dots,v_t)$ in $\calD_i(\Gamma_m)$, and denote by $G_{v_0v_1\dots v_t}$ the point-wise stabilizer of this arc in $G$. If $G_{v_0v_1\dots v_t}$ acts transitively on the out-neighborhood of $v_t$, then $G$ is $(t+1)$-arc-transitive on $\calD_i(\Gamma_m)$, contradicting the maximality of $t$. Hence $G_{v_0v_1\dots v_t}$ fixes each of the two out-neighbours of $v_t$. Since $\calD_i(\Gamma_m)$ is strongly connected, we then have $G_{v_0v_1\dots v_t}=1$ and so $|G_{v_0}|=2^t$. As Proposition~\ref{prop5} shows $|G_{v_0}|=2^{m+6}$, it follows that $t=m+6$. Thus the proposition holds.
\end{proof}

We are now in a position to prove Theorem~\ref{thm1}.

\begin{proof}[Proof of Theorem~\ref{thm1}]
Let $H$, $x$, $y$, $z$ be as in Section~\ref{sec2} and $\Gamma_m$ as in Construction~\ref{con1}. Note that $\Alt(H)\cong\A_{2^{m+6}}$ and $\Alt(H)_1\cong\A_{2^{m+6}-1}$. Then Theorem~\ref{thm1}(a) follows from Propositions~\ref{prop2} and~\ref{prop3}, and Theorem~\ref{thm1}(c) follows from Proposition~\ref{prop4}. Moreover, Proposition~\ref{prop5} asserts parts~(b) and~(d) of Theorem~\ref{thm1}. In particular, $\Gamma_m$ is a tetravalent half-arc-transitive graph, which gives rise to two digraphs $\calD_1(\Gamma_m)$ and $\calD_2(\Gamma_m)$ of opposite orientations. By Proposition~\ref{prop6}, $\calD_1(\Gamma_m)$ and $\calD_2(\Gamma_m)$ satisfy part~(e) of Theorem~\ref{thm1}. This completes the proof.
\end{proof}

\vskip0.1in
\noindent\textsc{Acknowledgement.} The author is in a deep debt of gratitude to the anonymous referees for their valuable suggestions that have helped to improve the paper.

\end{document}